\documentclass[a4paper,twoside]{amsart}

\usepackage{amssymb}
\usepackage{ifpdf}
\ifpdf
  \usepackage[colorlinks,linkcolor=red,anchorcolor=blue,citecolor=green]{hyperref}
\else
\fi

\ifx\isdraft\undefined
\else
    \usepackage{fancyhdr}
\fi

\numberwithin{equation}{section}

\newtheorem*{thmRT}{Ramsey's Theorem}

\newtheorem{theorem}{Theorem}[section]

\newtheorem{claim}[theorem]{Claim}

\newtheorem{corollary}[theorem]{Corollary}

\newtheorem{lemma}[theorem]{Lemma}

\newtheorem{proposition}[theorem]{Proposition}
\newtheorem{question}[theorem]{Question}

\theoremstyle{definition}
\newtheorem{definition}[theorem]{Definition}

\newcommand{\uh}{\upharpoonright}
\renewcommand{\>}{\rangle}

\newcommand{\low}{\operatorname{low}}

\newcommand{\RCA}{\operatorname{RCA}_0}
\newcommand{\ACA}{\operatorname{ACA}_0}
\newcommand{\WKL}{\operatorname{WKL}_0}

\newcommand{\RT}{\operatorname{RT}}
\newcommand{\COH}{\operatorname{COH}}
\newcommand{\SRT}{\operatorname{SRT}}

\newcommand{\ADS}{\operatorname{ADS}}

\newcommand{\CAC}{\operatorname{CAC}}

\newcommand{\ART}{\operatorname{ART}}
\newcommand{\FS}{\operatorname{FS}}
\newcommand{\TS}{\operatorname{TS}}
\newcommand{\RRT}{\operatorname{RRT}}

\begin{document}

\title{Some logically weak Ramseyan theorems}


\author{Wei Wang}

\thanks{This research is partially supported by NSF Grant 11001281 of China and an NCET grant from the Ministry of Education of China. The author thanks Carl Jockusch and Chitat Chong for their comments on some details in an earlier version. He also thanks the referees for their many helpful suggestions.}

\address{Institute of Logic and Cognition and Department of Philosophy, Sun Yat-sen University, 135 Xingang Xi Road, Guangzhou 510275, P.R. China}
\email{wwang.cn@gmail.com}

\begin{abstract}
We study four families of consequences of Ramsey's Theorem from the viewpoint of reverse mathematics. The first, which we call the Achromatic Ramsey Theorem, is from a partition relation introduced by Erd\H{o}s, Hajnal and Rado: $\omega \to [\omega]^r_{c,\leq d}$, which asserts that for every $f: [\omega]^r \to c$ there exists an infinite $H$ with $|f([H]^r)| \leq d$. The second and third are the Free Set Theorem and the Thin Set Theorem, which were introduced by Harvey Friedman. And the last is the Rainbow Ramsey Theorem. We show that, most theorems from these families are quite weak, i.e., they are strictly weaker than $\ACA$ over $\RCA$. Interestingly, these families turn out to be closely related. We establish the so-called strong cone avoidance property of most instances of the Achromatic Ramsey Theorem by an induction of exponents, then apply this and a similar induction to obtain the strong cone avoidance property of the Free Set Theorem. From the strong cone avoidance property of the Achromatic Ramsey Theorem and the Free Set Theorem, we derive the strong cone property of the Thin Set Theorem and the Rainbow Ramsey Theorem. It follws easily that a theorem with the strong cone avoidance property does not imply $\ACA$ over $\RCA$.
\end{abstract}

\ifx\isdraft\undefined
\else
    \today
\fi

\maketitle

\section{Introduction}\label{s:Introduction}

Reverse mathematics of Ramsey theory has been an active subject for computability theorists for years, in which Ramsey's Theorem for pairs ($\RT^2_2$) has enjoyed being the focus, perhaps since the work of Jockusch \cite{Jockusch:1972.Ramsey}. To facilitate the following discussion of Ramsey theory, let us recall some terminology. If $X$ is a set and $0 < r < \omega$, then $[X]^r$ is the set of $r$-element subsets of $X$; when we write $[X]^r$, $r$ is always a positive integer. A function $f$ is also called a \emph{coloring} or a \emph{partition}, and its values are naturally called \emph{colors}. A \emph{finite coloring} or \emph{$c$-coloring} is a function with finite range or with range contained in $c = \{0,1,\ldots,c-1\}$ where $c$ is a positive integer. For a finite coloring $f: [\omega]^r \to c$, a set $H$ is \emph{homogeneous for $f$} if $f$ is constant on $[H]^r$.

\begin{thmRT}
Every $f: [\omega]^r \to c$ for positive $c$ and $r > 1$ admits an infinite $f$-homogeneous set.
\end{thmRT}

For a fixed pair $r$ and $c$, $\RT^r_c$ is the instance of Ramsey's Theorem for all $c$-colorings of $r$-tuples.

In \cite{Jockusch:1972.Ramsey}, Jockusch conjectured that some computable two coloring of pairs may have only infinite homogeneous sets computing the halting problem. In the language of reverse mathematics, we may formulate Jockusch's conjecture as: $\RCA + \RT^2_2 \vdash \ACA$. This conjecture was later refuted by Seetapun \cite{Seetapun.Slaman:1995.Ramsey}. In his ingenious proof, Seetapun exploited the power of $\Pi^0_1$ classes in controlling complexity, which is encapsulated in a theorem of Jockusch and Soare \cite{Jockusch.Soare:1972.TAMS}. Seetapun's theorem was later significantly strengthened by Cholak, Jockusch and Slaman \cite{Cholak.Jockusch.ea:2001.Ramsey}. Cholak, Jockusch and Slaman \cite{Cholak.Jockusch.ea:2001.Ramsey} introduced two new ideas which have proven fruitful (e.g., see \cite{Hirschfeldt.Shore:2007}) and also apply to our work here. The first is the decomposition of $\RT^2_2$ to the so-called $\COH$ and $\SRT^2_2$; and the second is replacing a slightly more complicated forcing notion in \cite{Seetapun.Slaman:1995.Ramsey} by Mathias forcing. In \cite{Seetapun.Slaman:1995.Ramsey, Cholak.Jockusch.ea:2001.Ramsey}, several questions were raised: whether $\RT^2_2$ implies $\WKL$; whether Ramsey's Theorem for stable $2$-colorings of pairs ($\SRT^2_2$) is equivalent to $\RT^2_2$; and whether $\RT^2_2$ implies $I\Sigma_2$. Of course, all these questions are based on $\RCA$, which is a base theory for most work in reverse mathematics. These had been major open questions in the reverse mathematics of Ramsey theory. The first two have been negatively answered by Jiayi Liu \cite{Liu:2012} and Chong, Slaman and Yang in \cite{Chong.Slaman.ea:2013.SRT}, respectively. More recently, Chong, Slaman and Yang announced a negative answer to the third question.

Besides these major open questions, various authors have also studied consequences of Ramsey's Theorem, mostly of $\RT^2_2$. Many consequences of $\RT^2_2$ have been shown to be strictly weaker; and relations between these consequences give rise to a complicated picture, which fits the tradition of computability theory quite well. For example, Hirschfeldt and Shore \cite{Hirschfeldt.Shore:2007} proved that the Ascending and Descending Sequences principle ($\ADS$) and the Chain and Antichain principle ($\CAC$) are both strictly weaker than $\RT^2_2$; Csima and Mileti \cite{Csima.Mileti:2009.rainbow} proved that the Rainbow Ramsey Theorem for pairs ($\RRT^2_2$) does not imply $\ADS$, and thus is also strictly weaker than $\RT^2_2$.

Now we turn to Ramsey theory with an arbitrary exponent, which seems to deserve more attention from computability theorists. But we do know something, in particular, about complexity bounds. The work of Jockusch \cite{Jockusch:1972.Ramsey} gives some important answers.

\begin{theorem}[Jockusch \cite{Jockusch:1972.Ramsey}]\label{thm:Jo.bounds}
Every computable finite coloring of $[\omega]^r$ admits an infinite homogeneous set in $\Pi^0_{r}$. On the other hand, for each $r$, there are computable $2$-colorings of $[\omega]^r$ which admit no infinite homogeneous sets in $\Sigma^0_r$.
\end{theorem}

It is known that the complexity bounds above also appear in various consequences of $\RT^r_2$. Cholak, Giusto, Hirst and Jockusch \cite{Cholak.Giusto.ea:2005.freeset} showed that the $\Pi^0_r/\Sigma^0_r$ bounds apply to the Free Set Theorem and the Thin Set Theorem; by Csima and Mileti \cite{Csima.Mileti:2009.rainbow}, these bounds apply to the Rainbow Ramsey Theorem; and by Chubb, Hirst and McNicholl \cite{Chubb.Hirst.ea:2009}, the same bounds apply to a binary tree version of Ramsey's Theorem.

From a provability viewpoint, we also have a few results of Ramsey theory. By Jockusch \cite{Jockusch:1972.Ramsey}, we learn that $\RT^3_2$ is equivalent to $\ACA$ over $\RCA$. Recently, the author has proved that $\RCA + \RRT^3_2 \not\vdash \ACA$ in \cite{Wang:RRT}; and later in \cite{Wang:2013.COH.RRT} obtained some strengthening that $\RCA + \RRT^3_2 \not\vdash \WKL$ and $\RCA + \RRT^3_2 \not\vdash \RRT^4_2$.

The aim of this paper is to study Ramsey theory of larger exponents, mainly from a provability viewpoint. We consider several families of consequences of Ramsey's Theorem. As usual, we take $\RCA$ as the base theory and may assume it without explicit reference.

The first family was introduced by Erd\H{o}s, Hajnal and Rado \cite{ErdHos.Hajnal.ea:1965.partition}:
$$
    \omega \to [\omega]^r_{c, < d}
$$
if and only if for every $c$-coloring $f$ of $[\omega]^r$, there exists an infinite $H$ such that $|f([H]^r)| < d$. We call these partition relations the \emph{Achromatic Ramsey Theorem} ($\ART$), and write $\ART^r_{c, d}$ for $\omega \to [\omega]^r_{c, < d + 1}$ and $\ART^r_{<\infty, d}$ for $\forall c \ART^r_{c,d}$.

The second and third families are the \emph{Free Set Theorem} and the \emph{Thin Set Theorem}, which were introduced by Harvey Friedman when he developed Boolean Relation Theory (see \cite{Friedman:BRT}). For a function $f: [\omega]^r \to \omega$, a set $H$ is \emph{free} if $f(x_0,\ldots,x_{r-1}) \not\in H - \{x_0,\ldots,x_{r-1}\}$ for all $(x_0,\ldots,x_{r-1}) \in [H]^r$; and a set $S$ is \emph{thin for $f$} if $f([S]^r) \neq \omega$. The Free Set Theorem ($\FS$) asserts that every $f: [\omega]^r \to \omega$ admits an infinite free set, and the Thin Set Theorem ($\TS$) asserts that every $f$ as above admits an infinite thin set. $\FS^r$ and $\TS^r$ are instances of $\FS$ and $\TS$ for fixed exponent $r$, respectively. Cholak et al. \cite{Cholak.Giusto.ea:2005.freeset} proved that $\RT^r_2$ implies $\FS^r$ and $\FS^r$ implies $\TS^r$.

The last family is the \emph{Rainbow Ramsey Theorem}. The Rainbow Ramsey Theorem concerns bounded colorings: a coloring $f: [\omega]^r \to \omega$ is \emph{$b$-bounded}, if $|f^{-1}(c)| \leq b$ for all $c$. A \emph{rainbow} for a coloring $f: [\omega]^r \to \omega$ is a set $H$ such that $f$ is injective on $[H]^r$. The Rainbow Ramsey Theorem ($\RRT$) asserts that for every pair of positive integers $b$ and $r$, every $b$-bounded coloring of $[\omega]^r$ admits an infinite rainbow, and $\RRT^r_b$ is the instance of the Rainbow Ramsey Theorem for fixed exponent $r$ and bound $b$. Galvin gave an easy proof of $\RRT^r_2$ from $\RT^r_2$ (see \cite{Csima.Mileti:2009.rainbow}), which can be easily translated to yield $\RCA \vdash \forall r(\RT^r_2 \to \RRT^r_2)$.

It turns out that these families are closely related and share the same logical strength. We show that, for positive integers $r$ and sufficiently large $d$, neither $\ART^r_{< \infty,d}$ nor $\FS^r$ implies $\ACA$. Thus $\TS \not\vdash \ACA$. Moreover, we show that $\FS^r \vdash \RRT^r_2$ and consequently $\RRT \not\vdash \ACA$. So, we negatively answer Question 7.6 in \cite{Cholak.Giusto.ea:2005.freeset} and Question 5.15 in \cite{Csima.Mileti:2009.rainbow}.

As one would expect, we establish the logical strength of the Achromatic Ramsey Theorem and the Free Set Theorem by proving some cone avoiding theorems and then building Turing ideals which do not contain the degree of the halting problem. From the point of view of model theory, it may be slightly more natural to consider this common method as building a model which omits certain types of second order variables.

\begin{definition}\label{def:Omitting}
Suppose that $\mathcal{C}$ is a subset of the reals, and $\Phi = \forall X \exists Y \varphi(X,Y)$ is a $\Pi^1_2$ sentence, where $\varphi$ is an arithmetic formula with second order parameters.
\begin{enumerate}
    \item For a fixed $X$, a set $Y$ with $\varphi(X,Y)$ is called a \emph{solution} of $\Phi$ with respect to $X$.
    \item If for every $X$ which computes no real in $\mathcal{C}$, there exists a solution $Y$ of $\Phi$ with respect to $X$ such that $X \oplus Y$ computes no real in $\mathcal{C}$ either, then we say that $\Phi$ admits \emph{$\mathcal{C}$-omitting}.
    \item If for any $X$, whether $X$ computes any real in $\mathcal{C}$ or not, there exists a solution $Y$ of $\Phi$ with respect to $X$ such that $Y$ computes no real in $\mathcal{C}$, then we say that $\Phi$ admits \emph{strong $\mathcal{C}$-omitting}.
    \item $\Phi$ has the \emph{(strong) cone avoidance property} if and only if
        $$
            \forall A, B (A \not\leq_T B \to \Phi \text{ admits (strong) $\mathcal{C}$-omitting for } \mathcal{C} = \{Z: A \leq_T B \oplus Z\}).
        $$
\end{enumerate}
\end{definition}

Suppose that $p$ is a type of one second order variable which cannot be satisfied by computable reals. If $\Phi$ admits omitting for the set of reals satisfying $p$, then we can build a countable $\omega$-model of $\RCA + \Phi$ omitting $p$. So, the cone avoidance property is sufficient for proving $\Phi \not\vdash \ACA$. But, most theorems in the four families enjoy the strong cone avoidance property. To be precise, for each $r$, $\ART^r_{< \infty,d}$ has the strong cone avoidance property for sufficiently large $d$; and $\FS^r, \TS^r$ and $\RRT^r_2$ all have the strong cone avoidance property. Actually, the strong cone avoidance property is a key factor which allows us to establish the cone avoidance property by induction on exponents for the Achromatic Ramsey Theorem and the Free Set Theorem. Interestingly, these inductions follow a zigzag pattern. For example, the induction for $\FS$ proceeds as follows: from the induction hypothesis that the strong cone avoidance property holds for $\FS^{< r}$, we obtain the cone avoidance property for $\FS^r$; and apply this weaker property to obtain the strong cone avoidance property for $\FS^r$.

The proofs of the strong cone avoidance property for the Achromatic Ramsey Theorem and the Free Set Theorem share some other common features. Both proofs follow the cohesive-stable decomposition in Cholak, Jockusch and Slaman \cite{Cholak.Jockusch.ea:2001.Ramsey}: we obtain some cone avoiding set which assumes a property similar to that of cohesiveness; then build a desired set as a subset of this cohesive-like set. And in both proofs, we use Mathias forcing and exploit the power of $\Pi^0_1$ classes in controlling the complexity of certain Mathias generics, as Seetapun did in his celebrated proof.

Besides these similarities, the proof of the strong cone avoidance for the Free Set Theorem heavily depends on the strong cone avoidance property of the Achromatic Ramsey Theorem. So, all clues here suggest that there are some deeper relations, perhaps metamathematical relations, between the Achromatic Ramsey Theorem and the Free Set Theorem. However, at this moment we know nothing. We pose related questions in the last section.

Below, we briefly introduce the remaining sections:
\begin{itemize}
  \item In \S \ref{s:Preliminaries}, we introduce some conventions to facilitate technical formulations, and also recall some useful known results and some basic properties of Mathias forcing.
  \item In \S \ref{s:WRT.sca}, we establish the strong cone avoidance property of the Achromatic Ramsey Theorem and the Thin Set Theorem, and also the logical strength of the Achromatic Ramsey Theorem and the Thin Set Theorem.
  \item In \S \ref{s:FS.sca}, we establish the strong cone avoidance property and the logical strength of the Free Set Theorem. In addition, we reduce the Rainbow Ramsey Theorem to the Free Set Theorem, and thus obtain the strong cone avoidance property of $\RRT$.
  \item In \S \ref{s:Questions}, we raise some questions.
\end{itemize}

As mentioned, all the reverse mathematical results below are based on $\RCA$ and will usually be stated without explicit reference to $\RCA$.

\section{Preliminaries}\label{s:Preliminaries}

In this section, we set up some conventions and recall some notions and known results which are useful for our purposes. For more background knowledge in computability and reverse mathematics, we refer the reader to \cite{Lerman:83} and \cite{Simpson:1999.SOSOA}. We also need some elementary facts about algorithmic randomness, which can be found in \cite{Nies:2010.book}.

\subsection{Sequences}

If $s$ and $t$ are two finite sequences, then we write $st$ for the concatenation of $s$ and $t$. If $x$ is a single symbol, then $\<x\>$ is the finite sequence with only one symbol $x$. The length of a finite sequence $s$ is denoted by $|s|$. If $l < |s|$ then $s \uh l$ is the initial segment of $s$ of length $l$. For $X \subseteq \omega$, $X \uh l$ is interpreted as an initial segment of the characteristic function of $X$ in the obvious way.

Recall that $[X]^r$ for $0 < r < \omega$ is the set of $r$-element subsets of $X$. We also write $[X]^\omega$ for the set of countable subsets of $X$; $[X]^{< r}, [X]^{\leq r}, [X]^{< \omega}, [X]^{\leq \omega}$ are interpreted naturally. If $X \subseteq \omega$, then elements of $[X]^{\leq \omega}$ are identified with strictly increasing sequences. We use $\sigma, \tau, \ldots$ for elements of $[\omega]^{< \omega}$. Under the above convention, we may perform both sequence operations and set operations on elements of $[\omega]^{<\omega}$. For example, we can write $\sigma\tau$ for $\sigma \cup \tau$, if $\max \sigma < \min \tau$; $\sigma \subseteq \tau$ if $\sigma$ is a subset of $\tau$; and $\sigma - \tau = \{x \in \sigma: x \not\in \tau\}$. We extend this convention to infinite subsets of $\omega$, so we write $\sigma X$ for $\sigma \cup X$, if $\max \sigma < \min X$ and $X \in [\omega]^{\leq \omega}$.

\subsection{Trees}

We work with trees which are subsets of $\omega^{<\omega}$. If $T$ is a finite tree, then
$$
    [T] = \{\sigma \in T: \forall x(\sigma\<x\> \not\in T)\};
$$
if $T$ is an infinite tree, then $[T]$ denotes the set of infinite sequences whose initial segments are always in $T$.

When we use finite trees for measure theoretic arguments, we define a function $m_T$ for each finite tree $T$ by induction: $m_T(\emptyset) = 1$, if $\sigma\<x\> \in T$ then
$$
    m_T(\sigma\<x\>) = \frac{m_T(\sigma)}{|\{y: \sigma\<y\> \in T\}|}.
$$
We should consider $m_T$ as a probability measure associated with $T$. So, we can naturally extend the domain of $m_T$ to include certain subsets of $T$, and denote the resulting function by $m_T$ too: if $S \subseteq T$ is \emph{prefix-free} (i.e., if $S$ contains $\sigma$ then $S$ contains \emph{no} proper initial segment of $\sigma$), typically $S \subseteq [T]$, then
$$
    m_T S = \sum_{\sigma \in S} m_T(\sigma).
$$

\subsection{Computation}

For a finite sequence $\sigma$, we write $\Phi_e(\sigma; x) \downarrow$ if $\Phi_e(\sigma; x)$ converges in $|\sigma|$ many steps. We write $\Phi_e(\sigma; x) \uparrow$ for $\neg (\Phi_e(\sigma; x) \downarrow)$.

For a set $B$ and a finite sequence $\sigma$, we write $\Phi_e^B(\sigma; x) \downarrow$ if $\Phi_e((B \uh |\sigma|) \oplus \sigma; x) \downarrow$. Notations, like $\Phi_e^B(\sigma; x) \uparrow$ and $\Phi_e^B(X)$, are interpreted in similar way.

To force a non-computability statement like $\Phi_e^B(H) \neq A$, splitting computations are usually helpful. A pair $(\eta_0,\eta_1) \in [\omega]^{<\omega} \times [\omega]^{<\omega}$ is \emph{$(e,B)$-splitting over $\sigma \in [\omega]^{<\omega}$}, if $\max \sigma < \min \eta_i$ for $i < 2$ and $\Phi_e^B(\sigma \eta_0; x) \downarrow \neq \Phi_e^B(\sigma \eta_1; x) \downarrow$ for some $x$.

\subsection{Some useful known results}

We list some useful results here, but formulate some of them in terms of Definition \ref{def:Omitting}.

By relativizing \cite[Corollary 2.11]{Jockusch.Soare:1972.TAMS}, we obtain the following theorem.

\begin{theorem}[Jockusch and Soare]\label{thm:JoSo}
$\WKL$ has the cone avoidance property.
\end{theorem}

Theorem \ref{thm:JoSo} reflects the power of $\Pi^0_1$ classes in controlling complexity, and plays an important role in Seetapun's proof of the following theorem (\cite[Theorem 2.1]{Seetapun.Slaman:1995.Ramsey}).

\begin{theorem}[Seetapun]\label{thm:Seetapun}
$\RT^2_2$ has the cone avoidance property.
\end{theorem}

Dzhafarov and Jockusch discovered a neglected feature of Seetapun's proof that the proof works for finite partitions of $\omega$ of arbitrary complexity (\cite[Lemma 5.2(i)]{Dzhafarov.Jockusch:2009}).

\begin{theorem}[Dzhafarov and Jockusch]\label{thm:DzJo}
The infinite pigeonhole principle has the strong cone avoidance property.
\end{theorem}

For the Free Set Theorem, we need the following theorem which follows from the proof of Theorem 5.2 in Cholak et al. \cite{Cholak.Giusto.ea:2005.freeset}.

\begin{theorem}[Cholak et al.]\label{thm:CGHJ}
For each $f: [\omega]^r \to \omega$, there exists $g: [\omega]^r \to 2r + 2$ such that $g \leq_T f$ and $g \oplus H$ computes an infinite $f$-free set for every infinite $g$-homogeneous $H$. Moreover, if $f(\sigma) \leq \max \sigma$ for all $\sigma \in [\omega]^r$ then every $g$-homogeneous set is $f$-free.
\end{theorem}

Note that, combining Theorems \ref{thm:CGHJ} and \ref{thm:DzJo}, if we restrict the Free Set Theorem for $f: \omega \to \omega$ such that $f(x) \leq x$ for all $\sigma \in [\omega]^r$, then we have the strong cone avoidance property.

\subsection{Mathias forcing}

Here we include a well-known computability theoretic property of Mathias forcing and also an easy corollary of this property that $\COH$ has the strong cone avoidance property.

\begin{definition}\label{def:M-forcing}
A \emph{Mathias condition} is a pair $(\sigma,X) \in [\omega]^{<\omega} \times [\omega]^\omega$ such that $\max \sigma < \min X$. We identify a Mathias condition $(\sigma,X)$ with the set below:
$$
    \{Y \in [\omega]^\omega: \sigma \subset Y \subseteq \sigma \cup X\}.
$$

For two Mathias conditions $(\sigma,X)$ and $(\tau,Y)$, $(\tau,Y) \leq_M (\sigma,X)$ if and only if $(\tau,Y) \subseteq (\sigma,X)$ under the above convention.
\end{definition}

\begin{lemma}[Folklore]\label{lem:M-forcing.ca}
For each $e$ and a Mathias condition $(\sigma,X)$ with $A \not\leq_T B \oplus X$, there exists a Mathias condition $(\tau,Y) \leq_M (\sigma,X)$ such that $A \not\leq_T B \oplus Y$ and $\Phi_e^B(Z) \neq A$ for every $Z \in (\tau,Y)$.
\end{lemma}

\begin{proof}
There are two cases.

\medskip

\emph{Case 1:} $X$ contains a pair $(\eta_0,\eta_1)$ which $(e,B)$-spits over $\sigma$.

Fix $i < 2$ and $x$ such that $\Phi^B_e(\sigma\eta_i; x) \downarrow \neq A(x)$. Let $\tau = \sigma\eta_i$ and $Y = X \cap (\max \eta_i, \infty)$. Then $(\tau,Y)$ is as desired.

\medskip

\emph{Case 2:} $X$ contains no pair $(e,B)$-splitting over $\sigma$.

If $Z \in [X]^\omega$ and $\Phi^B(Z)$ is total then $\Phi^B(Z) \leq_T B \oplus X$, and thus $\Phi^B(Z) \neq A$, since $A \not\leq_T B \oplus X$. So we can simply let $(\tau,Y) = (\sigma,X)$.
\end{proof}

The following theorem is an easy corollary of the above lemma and Theorem \ref{thm:DzJo}. Recall that an infinite set $C$ is \emph{cohesive} for a sequence $\vec{R} = (R_n: n < \omega)$, if and only if for each $n$ either $C \cap R_n$ or $C - R_n$ is finite. $\COH$, a consequence of $\RT^2_2$ introduced by Cholak, Jockusch and Slaman \cite[Statement 7.7]{Cholak.Jockusch.ea:2001.Ramsey}, asserts that every sequence admits a cohesive set.

\begin{theorem}[Folklore]\label{thm:COH.sca}
$\COH$ has the strong cone avoidance property.
\end{theorem}

\section{The Achromatic Ramsey Theorem}\label{s:WRT.sca}

In this section, we prove that $\ART^r_{<\infty,d}$ has the strong cone avoidance property for appropriate $d$.

\begin{theorem}\label{thm:WRT.sca}
For each $r > 0$, there exists $d$ such that $\ART^r_{<\infty,d}$ has the strong cone avoidance property. Hence, $\ART^r_{<\infty,d} \not\vdash \ACA$ for sufficiently large $d$.
\end{theorem}

Clearly, the second part of Theorem \ref{thm:WRT.sca} is a consequence of the first part. Before proving the first part of Theorem \ref{thm:WRT.sca}, we present some easy corollaries.

\begin{theorem}\label{thm:TS}
$\TS$ has the strong cone avoidance property. Thus $\TS \not\vdash \ACA$.
\end{theorem}

\begin{proof}
Fix $X \not\leq_T Y$ and $f: [\omega]^r \to \omega$ with $r > 0$. By the strong cone avoidance of the Achromatic Ramsey Theorem, let $d$ be such that $\ART^r_{d+1,d}$ has the strong cone avoidance property. For each $\sigma \in [\omega]^r$, let $g(\sigma) = \min\{d,f(\sigma)\}$. So, $g: [\omega]^r \to d+1$. Pick $Z \in [\omega]^\omega$ such that $X \not\leq Y \oplus Z$ and $|g([Z]^r)| \leq d$. Then $Z$ is clearly thin for $f$. So, $\TS$ has the strong cone avoidance property.

Hence, $\TS \not\vdash \ACA$.
\end{proof}

As with many other consequences of Ramsey's Theorem, the Achromatic Ramsey Theorem also obeys the bounds of Jockusch in Theorem \ref{thm:Jo.bounds}.

\begin{proposition}\label{prp:WRT.bounds}
Fix $r \geq 2$, $c \geq 2$ and $d > 0$.
\begin{enumerate}
 \item For each computable $f: [\omega]^r \to c$, there exists an infinite $H \in \Pi^0_r$ such that $|f([H]^r)| \leq d$.
 \item There exists a computable $g: [\omega]^r \to d+1$ such that no infinite $H \in \Sigma^0_r$ can have $|f([H]^r)| \leq d$.
\end{enumerate}
\end{proposition}

\begin{proof}
(1) follows easily from Theorem \ref{thm:Jo.bounds}.

On the other hand, Cholak et al. \cite[Theorem 4.1]{Cholak.Giusto.ea:2005.freeset} defined a computable $h: [\omega]^r \to \omega$ which admits no infinite thin set in $\Sigma^0_r$. So (2) follows from this known bound and the proof of Theorem \ref{thm:TS}.
\end{proof}

\begin{corollary}
For $r > 2$ and $c > d > 0$, $\RT^2_2 \not\vdash \ART^r_{c,d}$. Consequently, $\ART^3_{4,3}$ is strictly between $\RT^2_2$ and $\RT^3_2$ and $\ART^3_{3,2}$ is strictly between $\RT^2_{<\infty}$ and $\RT^3_2$, where $\RT^2_{<\infty}$ is $\forall n \RT^2_n$.
\end{corollary}

\begin{proof}
By relativizing Theorem 3.1 of Cholak, Jockusch and Slaman \cite{Cholak.Jockusch.ea:2001.Ramsey}, there exists an $\omega$-model $\mathcal{M}$ of $\RCA + \RT^2_2$ containing only $\Delta^0_3$ sets. By Proposition \ref{prp:WRT.bounds}(2), $\mathcal{M} \not\models \ART^r_{c,d}$.

The above $\omega$-model is also a model of $\RT^2_{<\infty}$. On the other hand, Dorais et al. \cite[\S 5]{Dorais.Dzhafarov.ea:2013} prove that $\ART^3_{4,3} \vdash \RT^2_2$ and $\ART^3_{3,2} \vdash \RT^2_{<\infty}$.
\end{proof}

We prove the first part of Theorem \ref{thm:WRT.sca} by induction on $r$, that $\ART^r_{<\infty,d}$ has the strong cone avoidance property for sufficiently large $d$. The induction proceeds in a zigzag way:
\begin{enumerate}
 \item[(A1)] As the infinite pigeonhole principle has the strong cone avoidance property, we get the strong cone avoidance property of $\ART^1_{<\infty, 1}$.
 \item[(A2)] Fix $(d_k: 0 < k < r)$ such that $\ART^k_{<\infty,d_k}$ has the strong cone avoidance property, for each $k \in (0,r)$. Firstly we prove that $\ART^r_{<\infty,d_{r-1}}$ has the cone avoidance property.
 \item[(A3)] Then we prove that $\ART^r_{<\infty,d}$ has the \emph{strong} cone avoidance property for
  $$
      d = d_{r-1} + \sum_{0 < k < r} d_k d_{r-k}.
  $$
\end{enumerate}

(A1) is trivial. (A2) is accomplished by the lemma below.

\begin{lemma}\label{lem:WRT.ca}
If $\ART^n_{c,e}$ has the strong cone avoidance property, then $\ART^{n+1}_{c,e}$ has the cone avoidance property.
\end{lemma}

\begin{proof}
Fix $X, Y$ and $g: [\omega]^{n+1} \to c$ such that $X \not\leq_T Y \oplus g$.

For each $\sigma \in [\omega]^n$ and $k < c$, let $R_{\sigma,k} = \{x: g(\sigma\<x\>) = k\}$. By the cone avoidance property of $\COH$, pick $Z$ such that $X \not\leq_T Y \oplus g \oplus Z$ and $Z$ is cohesive for $(R_{\sigma,k}: \sigma \in [\omega]^n, k < c)$. For each $\sigma \in [\omega]^n$, let $\bar{g}(\sigma) = \lim_{x \in Z} g(\sigma\<x\>)$, which is defined by the cohesiveness of $Z$. By the strong cone avoidance of $\ART^n_{c,e}$, pick $W \in [Z]^\omega$ such that $X \not\leq_T Y \oplus g \oplus W$ and $|\bar{g}([W]^n)| \leq e$. Let $\theta = \bar{g}([W]^n)$.

We build a strictly increasing sequence $(\sigma_s \in [W]^{<\omega}: s < \omega)$ by induction. Let $\sigma_0 = \emptyset$. Suppose that $\sigma_s \in [W]^{<\omega}$ and $g([\sigma_s]^{n+1}) \subseteq \theta$. As $\bar{g}([\sigma_s]^n) \subseteq \theta$, $g(\rho \<x\>) = \bar{g}(\rho) = \lim_{s \in W} g(\rho \<s\>) \in \theta$ for all $\rho \in [\sigma_s]^n$ and sufficiently large $x \in W$. So, in a $g \oplus W$-computable way, we can pick
$$
    x_s = \min \{x \in W: x > \max \sigma_s \wedge \forall \rho \in [\sigma_s]^n(g(\rho\<x\>) \in \theta)\}.
$$
Let $\sigma_{s+1} = \sigma_s \<x\>$.

So, $V = \bigcup_s \sigma_s$ is $g \oplus W$-computable and infinite, and $g([V]^{n+1}) \subseteq \theta$. Moreover, $X \not\leq_T Y \oplus g \oplus V$, as $X \not\leq_T Y \oplus g \oplus W$.
\end{proof}

The remaining part of this section is devoted to proving (A3). We fix $A \not\leq_T B$ and $f: [\omega]^r \to c$ where $c < \omega$. In what follows, we say that a set $X$ is \emph{cone avoiding} if $A \not\leq_T B \oplus X$; and a Mathias condition $(\sigma,X)$ is \emph{cone avoiding} if $X$ is cone avoiding. Tentatively, we say that $(\eta_n: n < \kappa)$ ($\kappa \leq \omega$) is \emph{an increasing block sequence}, if $\eta_n \in [\omega]^{<\omega}$ and $\max \eta_n < \min \eta_{n+1}$ for $n+1 < \kappa$.

We need a cone avoiding infinite $H$ such that $|f([H]^r)| \leq d$. The strategy is to build $H$ as the union of an increasing block sequence $(\eta_n: n < \omega)$ such that
\begin{enumerate}
    \item[(H1)] if $\tau \in [\eta_n]^r$ for some $n$ then $f(\tau) \in \theta$ for some fixed $\theta \in [c]^{\leq d_{r-1}}$;
    \item[(H2)] if $\tau$ is an $r$-tuple with a non-empty proper initial segment $\rho$ contained in some $\eta_n$ and the remaining final segment contained in $\bigcup_{m > n} \eta_m$, then $f(\tau) \in \theta_{\rho}$ for some $\theta_{\rho} \in [c]^{\leq d_{r-k}}$, where $k = |\rho|$; moreover, for each $k \in (0, r)$ there are at most $d_k$ many distinct $\theta_{\rho}$'s, i.e., $|\{\theta_\rho: \rho \in [H]^k\}| \leq d_k$.
\end{enumerate}

By (H1), if $\tau$ is contained in some $\eta_n$ then there are at most $d_{r-1}$ many possibilities for $f(\tau)$; otherwise, by (H2), there are at most $\sum_{0 < k < r} d_k d_{r-k}$ many possibilities for $f(\tau)$. So, $|f([H]^r)| \leq d$.

For each $k < r$ and $\rho \in [\omega]^{k}$, let $f_\rho(\tau) = f(\rho\tau)$ for $\tau \in [\omega]^{r-k}$ with $\max \rho < \min \tau$. We implement the above strategy in several steps:
\begin{enumerate}
    \item By the induction hypothesis and Mathias forcing, we build a cone avoiding $D \in [\omega]^\omega$ and a sequence $(\Theta_k: 0 < k < r)$ such that
        \begin{enumerate}
            \item each $\Theta_k$ is a set of at most $d_k$ many sets of colors, and $|\theta| \leq d_{r-k}$ for each $\theta \in \Theta_k$;
            \item if $0 < k < r$ and $\rho \in [D]^k$, then there exist $\theta_\rho \in \Theta_k$ so that $f_\rho(\tau) \in \theta_\rho$ for all $\tau \in [D]^{r-k}$ with $\min \tau$ sufficiently large.
        \end{enumerate}
    \item By a Seetapun-style Mathias forcing, we build a cone avoiding $G \in [D]^\omega$ as the union of an increasing block sequence $(\xi_n \in [D]^{<\omega}: n < \omega)$ such that
    \begin{enumerate}
        \item $|f([\xi_n]^r)| \leq d_{r-1}$;
        \item $f_\rho(\tau) \in \theta_\rho \in \Theta_k$, for $k \in (0,r)$, $\rho \in [\bigcup_{i < n} \xi_i]^k$ and $\tau \in [\bigcup_{i \geq n} \xi_i]^{r - k}$.
    \end{enumerate}
    \item By the strong cone avoidance of the infinite pigeonhole principle, we select $\eta_n$'s from some appropriate $\xi_i$'s and thus build $H$ as a subset of $G$.
\end{enumerate}

Step (2) in the above plan is the key step, in which we establish (H2) in a slightly stronger form and also make some progress towards (H1). The reader may wonder why we mix the above two tasks together in (2). An attempt to separate these two tasks would lead to a plan like the following:
\begin{enumerate}
    \item[(1')] We first build an increasing block sequence $(\xi_n: n < \omega)$ such that $G = \bigcup_n \xi_n$ is cone avoiding and $(\xi_n: n < \omega)$ satisfies (H2) in place of $(\eta_n: n < \omega)$;
    \item[(2')] Then we build $(\eta_n: n < \omega)$ satisfying (H1) with $\eta_n \subset G$.
\end{enumerate}
The difficulty in this simpler plan is that with the technique in this paper we would only make $\eta_n \subset G$ in (2') but not $\eta_n \subseteq \xi_i$ for some $i$. So in (2') we might lose (H2).

\subsection{The construction of $D$}\label{ss:WRT.sca.D}

Firstly, we build a cone avoiding $C \in [\omega]^\omega$ and a sequence $(\theta_\rho: 0 < |\rho| < r)$ such that for each $k \in (0,r)$ and $\rho \in [\omega]^k$,
\begin{enumerate}
    \item[(C1)] $\theta_\rho$ is a subset of $c$ with at most $d_{r-k}$ many elements;
    \item[(C2)] $f_\rho(\tau) \in \theta_\rho$ for all $\tau \in [C]^{r-k}$ with $\min \tau$ sufficiently large.
\end{enumerate}
Note that (C2) implies that $C$ has some kind of cohesiveness. Thus, it is not surprising that the construction of $C$ resembles constructions of cohesive sets.

\begin{lemma}
Suppose that $0 < k < r$ and $\rho \in [\omega]^k$. Then every cone avoiding Mathias condition $(\sigma,X)$ can be extended to another cone avoiding $(\sigma,Y)$ such that $\max \rho < \min Y$ and $|f_\rho([Y]^{r-k})| \leq d_{r-k}$.
\end{lemma}

\begin{proof}
As $\ART^{r-k}_{c,d_{r-k}}$ has the strong cone avoidance property, we can pick a cone avoiding $Y \in [X]^\omega$ such that $\max \rho < \min Y$ and $|f_\rho([Y]^{r-k})| \leq d_{r-k}$.
\end{proof}

With the above lemma and Lemma \ref{lem:M-forcing.ca}, we can obtain a descending sequence of cone avoiding Mathias conditions $((\sigma_n, X_n): n < \omega)$ and a sequence $(\theta_\rho: 0 < |\rho| < r)$, satisfying the following properties:
\begin{enumerate}
    \item If $k \in (0,r)$ and $\rho \in [\omega]^k$, then $\theta_\rho \in [c]^{\leq d_{r-k}}$ and there exists $n$ such that $f_\rho([X_n]^{r-k}) = \theta_\rho$;
    \item For each $n$, $|\sigma_n| < |\sigma_{n+1}|$ and $\Phi^B_n(Z) \neq A$ for all $Z \in (\sigma_n,X_n)$.
\end{enumerate}
So (C1) and (C2) hold for $C = \bigcup_n \sigma_n$ and $(\theta_\rho: 0 < |\rho| < r)$, and $C$ is cone avoiding.

Secondly, we build the desired $D \in [C]^\omega$. For each $k \in (0, r)$, define $F_k(\rho) = \theta_\rho$ for $\rho \in [C]^k$. As $\ART^k_{<\infty, d_k}$ has the strong cone avoidance property for each $k \in (0, r)$, we can obtain a sequence $(D_k: k < r)$ such that
\begin{enumerate}
    \item $D_0 = C$ and $D_{k+1} \in [D_k]^\omega$ is cone avoiding for each $k < r - 1$.
    \item $|F_k([D_{k}]^k)| \leq d_k$ if $0 < k < r$.
\end{enumerate}

Let $D = D_{r-1}$. For each $k \in (0, r)$, let $\Theta_k = F_k([D]^k)$. It follows that
\begin{itemize}
    \item[(D)] if $0 < k < r$ and $\rho \in [D]^k$, then $f(\rho\tau) \in \theta_\rho \in \Theta_k$ for all $\tau \in [D]^{r-k}$ with $\min \tau$ sufficiently large.
\end{itemize}

\subsection{The construction of $G$}\label{ss:WRT.sca.G}

In this subsection, we start with $D$ from \S \ref{ss:WRT.sca.D} and build an increasing block sequence $(\xi_n \in [D]^{<\omega}: n < \omega)$ such that
\begin{enumerate}
    \item[(G1)] $G = \bigcup_n \xi_n$ is infinite and cone avoiding;
    \item[(G2)] $|f([\xi_n]^r)| \leq d_{r-1}$;
    \item[(G3)] If $k \in (0,r)$ and $\rho \in [\bigcup_{i < n} \xi_i]^k$ then $f(\rho\tau) \in \theta_\rho$ for all $\tau \in [\bigcup_{i \geq n} \xi_i]^{r - k}$.
\end{enumerate}

Note that, if we ignore (G1) then we can easily get some $(\zeta_n \in [D]^{<\omega}: n < \omega)$ satisfying (G2) and (G3) in place of $(\xi_n: n < \omega)$. We start with $(\sigma_0, X_0) = (\emptyset, D)$, and extend $(\sigma_n, X_n)$ to $(\sigma_n, Y_{n+1})$ so that $f_\rho(\tau) = f(\rho\tau) \in \theta_\rho$ for $k \in (0,r)$, $\rho \in [\sigma_n]^k$ and $\tau \in [Y_{n+1}]^{r-k}$, then we extend $(\sigma_n, Y_{n+1})$ to $(\sigma_{n+1}, X_{n+1})$ with $\sigma_{n+1} = \sigma_n \zeta_n$ for some $\zeta_n$ of length $1$. By (D), we can even make $X_{n+1} = X_n \cap (b,\infty)$ for some $b$. However, in general we need $(f \oplus D)'$ to find such a lower bound $b$, thus we cannot ensure that $\bigcup_n \zeta_n$ is cone avoiding. So, the non-trivial job is to satisfy (G2, G3) and (G1) simultaneously.

To this end, we prove the following lemma, which helps us in extending a Mathias condition $(\sigma,X)$ to some $(\sigma\xi, Y)$ such that $|f([\xi]^r)| \leq d_{r-1}$ and $(\sigma\xi, Y)$ forces a cone avoiding requirement $\Phi_e^B(G) \neq A$. To prove this lemma, we follow an idea of Seetapun from \cite{Seetapun.Slaman:1995.Ramsey}. Although Seetapun's proof is well known, it appears tricky at first sight. So, we explain the key idea of the following proof below:
\begin{enumerate}
    \item To force a cone avoiding requirement, we search for splitting computations in a cone avoiding Mathias condition $(\sigma, X)$, as we did in the proof of Lemma \ref{lem:M-forcing.ca}. In addition, desirable splitting computations should have oracles obeying (G2). The existence of such splitting computations is a $\Sigma^{B \oplus X \oplus f}_1$ question $\varphi$. As $f$ is of arbitrary complexity, we cannot afford to ask $\varphi$ directly.
    \item To work around the above difficulty, we define a compact $\Pi^0_1$ class $\mathcal{C}$ of $c$-colorings which contains $f$. Instead of asking $\varphi$ which concerns the specific $f$, we ask whether for every coloring $g$ in $\mathcal{C}$ there exist splitting computations with oracles meeting a $g$-version of (G2). This question $\psi$ is $\Sigma^{B \oplus X}_1$, by the compactness of $\mathcal{C}$.
    \item If $\psi$ has a positive answer, then we can obtain an appropriate $\xi$ and simultaneously meet a cone avoiding requirement $\Phi^B_e(G) \neq A$.
    \item If $\psi$ has a negative answer, then we can apply the cone avoidance property of $\WKL$ and the induction hypothesis to get some cone avoiding $(\sigma, Y) \leq_M (\sigma, X)$, which forces $\Phi^B_e(G) \neq A$. Then $\xi = \emptyset$, which satisfies (G2) trivially.
\end{enumerate}

\begin{lemma}\label{lem:WRT.sca.G.ext}
For each $e$ and each cone avoiding Mathias condition $(\sigma,X)$, there exists a cone avoiding $(\sigma\xi,Y) \leq_M (\sigma,X)$ such that $|f([\xi]^{r})| \leq d_{r-1}$ and $\Phi^B_e(Z) \neq A$ for all $Z \in (\sigma\xi,Y)$.
\end{lemma}

\begin{proof}
Let $\mathcal{C}$ be the set of all $c$-colorings of $[\omega]^r$. Then $f \in \mathcal{C}$ and $\mathcal{C}$ is a compact $\Pi^0_1$ class. Let $\mathcal{U}$ be the set of $g \in \mathcal{C}$ such that if $\tau \in [X]^{<\omega}$ and $|g([\tau]^{r})| \leq d_{r-1}$ then $\tau$ contains \emph{no} pair $(e,B)$-splitting over $\sigma$. So, $\mathcal{U}$ is $\Pi^0_1$ in $B \oplus X$. (Now, $\psi$ is the question whether $\mathcal{U}$ is empty.)

\medskip

\emph{Case 1:} $\mathcal{U} = \emptyset$. In particular, $f \not\in \mathcal{U}$.

By the definition of $\mathcal{U}$, we can pick $\xi_0$ and $\xi_1$ from $[X]^{<\omega}$ and $x$ so that $|f([\xi_i]^r)| \leq d_{r-1}$ for $i < 2$ and $\Phi^B_e(\sigma\xi_0; x) \downarrow \neq \Phi^B_e(\sigma\xi_1; x) \downarrow$. Fix $i < 2$ such that $\Phi^B_e(\sigma\xi_i; x) \neq A(x)$ and let $\xi = \xi_i$. So, $(\sigma\xi, X \cap (\max \xi_i, \infty))$ is a cone avoiding extension as desired.

\medskip

\emph{Case 2:} $\mathcal{U} \neq \emptyset$.

As $X$ is cone avoiding, by the cone avoidance property of $\WKL$ (Theorem \ref{thm:JoSo}) there exists $g \in \mathcal{U}$ with $X \oplus g$ cone avoiding. By Lemma \ref{lem:WRT.ca} and the induction hypothesis that $\ART^{r-1}_{c,d_{r-1}}$ has the strong cone avoidance property, pick $Y \in [X]^\omega$ such that $Y$ is cone avoiding and $|g([Y]^r)| \leq d_{r-1}$. As $g \in \mathcal{U}$, $Y$ contains no pair $(e,B)$-splitting over $\sigma$. So, if $Z \in (\sigma, Y)$ and $\Phi^B_e(Z)$ is total then $\Phi^B_e(Z) \leq_T B \oplus Y$ and thus $\Phi^B_e(Z) \neq A$. Thus, $(\sigma,Y)$ is the desired extension.
\end{proof}

By the construction of $C$, every Mathias condition $(\sigma,X)$ with $X \subseteq C$ can be extended to some $(\tau,Y) = (\sigma, X \cap (b,\infty))$ such that $f_\rho(\upsilon) \in \theta_\rho$ for all non-empty $\rho \in [\tau]^{< r}$ and $\upsilon \in [Y]^{r - |\rho|}$.

By the above remark and Lemma \ref{lem:WRT.sca.G.ext}, we can build a descending sequence of cone avoiding Mathias conditions $((\sigma_n, X_n): n < \omega)$ such that
\begin{enumerate}
    \item $(\sigma_0, X_0) = (\emptyset, D)$;
    \item $f_\rho(\tau) \in \theta_\rho$ for all non-empty $\rho \in [\sigma_n]^{< r}$ and $\tau \in [X_n]^{r - |\rho|}$;
    \item $\sigma_{n+1} = \sigma_n \xi_n$ for some non-empty $\xi_n$ with $|f([\xi_n]^r)| \leq d_{r-1}$;
    \item $\Phi^B_n(Z) \neq A$ for all $Z \in (\sigma_{n+1},X_{n+1})$.
\end{enumerate}
Let $G = \bigcup_n \xi_n = \bigcup_n \sigma_n$. Then (G1-3) are satisfied.

\subsection{The construction of $H$}\label{ss:WRT.sca.H}

For each $n$, let $\alpha_n = f([\xi_n]^r)$. Then $\alpha_n$ is a subset of $c$ with at most $d_{r-1}$ many elements. For each $\alpha \in [c]^{\leq d_{r-1}}$, let
$$
    G_\alpha = \{x \in G: \exists n (x \in \xi_n \wedge \alpha_n = \alpha)\}.
$$
By the strong cone avoidance property of the infinite pigeonhole principle (Theorem \ref{thm:DzJo}), there exist $\alpha \in [c]^{\leq d_{r-1}}$ and a cone avoiding $H \in [G_\alpha]^\omega$.

\begin{lemma}
$H$ is the union of some increasing block sequence $(\eta_n: n < \omega)$ satisfying (H1, H2). Hence $|f([H]^r)| \leq d$.
\end{lemma}

\begin{proof}
Let $(\eta_n: n < \omega)$ be the sequence of non-empty $H \cap \xi_m$'s. Then $H = \bigcup_n \eta_n$.

As $\eta_n \subseteq \xi_m$ for some $m$ and $\xi_m$ satisfies (G2), (H1) holds. Suppose that $\sigma \in [H]^r$ is of the form $\rho\tau$ for some $\rho \subseteq \eta_n$ and $\tau \subset \bigcup_{m > n} \eta_m$. By (G3), $f(\sigma) = f_\rho(\tau) \in \theta_\rho$; by (D), $\theta_\rho \in \Theta_k$ where $k = |\rho|$. By the construction of $\theta_\rho$ and $\Theta_k$, $|\theta_\rho| \leq d_{r-k}$ and $|\Theta_k| \leq d_k$. So, (H2) holds as well.

Hence, $|f([H]^r)| \leq d$ by the remark following (H2).
\end{proof}

This completes the proof of Theorem \ref{thm:WRT.sca}.

\section{The Free Set Theorem}\label{s:FS.sca}

In this section, we establish the strong cone avoidance property for the Free Set Theorem of arbitrary finite exponent.

\begin{theorem}\label{thm:FS.sca}
$\FS$ has the strong cone avoidance property. Hence, $\FS \not\vdash \ACA$.
\end{theorem}

Before proving Theorem \ref{thm:FS.sca}, we apply it to obtain similar results for the Rainbow Ramsey Theorem.

\begin{theorem}\label{thm:RRT}
For each $n > 0$ and a $2$-bounded function $f$ on $[\omega]^n$, there exists a uniformly $f$-computable $g: [\omega]^n \to \omega$ such that every $g$-free set is an $f$-rainbow.

Hence, $\RRT$ has the strong cone avoidance property, $\RCA \vdash \forall n > 0(\FS^n \to \RRT^n_k)$ for every $k < \omega$ and $\RRT \not\vdash \ACA$.
\end{theorem}

\begin{proof}
By Theorem \ref{thm:FS.sca}, it suffices to prove the first half.

Fix a computable bijection $\ulcorner \cdot \urcorner: [\omega]^n \to \omega$. Let $f: [\omega]^n \to \omega$ be $2$-bounded. For each $\sigma \in [\omega]^n$, let
$$
    g(\sigma) =
    \left\{
      \begin{array}{ll}
        \min (\tau - \sigma), & \exists \tau (\ulcorner \tau \urcorner < \ulcorner \sigma \urcorner \wedge f(\sigma) = f(\tau)); \\
        0, & \hbox{otherwise.}
      \end{array}
    \right.
$$
As $f$ is $2$-bounded, if $\tau$ in the definition of $g(\sigma)$ exists then it is unique. As $\tau$ and $\sigma$ are two distinct finite sets of same size, $\tau - \sigma \neq \emptyset$. Thus $g$ is well defined and total. By $\FS^n$, let $X \in [\omega]^\omega$ be $g$-free.

We claim that $X$ is a rainbow for $f$. Assume that $f(\sigma) = f(\tau)$ for distinct $\sigma,\tau \in [X]^n$. Without loss of generality, assume that $\ulcorner \tau \urcorner < \ulcorner \sigma \urcorner$. Then, $g(\sigma) \in \tau - \sigma \subset X - \sigma$, and we have a desired contradiction.
\end{proof}

Below, we prove Theorem \ref{thm:FS.sca}. Clearly, the second part is a consequence of the first part. To prove the first part, the overall plan is to establish the strong cone avoidance property for $\FS^r$ by induction on the exponent $r$:
\begin{itemize}
  \item[(F1)] Firstly we prove that $\FS^1$ has the strong cone avoidance property;
  \item[(F2)] Then we establish the cone avoidance property for $\FS^r$ with $r > 1$, with the induction hypothesis that $\FS^{r-1}$ has the strong cone avoidance property;
  \item[(F3)] Finally we prove that $\FS^r$ has the \emph{strong} cone avoidance property for $r > 1$, with the full induction hypothesis for all lesser exponents.
\end{itemize}

A key idea to accomplish (F1) and (F3) is to reduce $\FS^r$ to $\FS^r$ for functions behaving tamely. We establish this reduction in Lemma \ref{lem:FS.sca.trapped} below.

For $r > 0$, each $\sigma \in [\omega]^r$ induces a finite sequence of \emph{traps} (i.e., intervals) $(I^\sigma_k: k \leq r)$, where
\begin{gather*}
    I^\sigma_0 = [0, \sigma(0)], \\
    I^\sigma_k = [\sigma(k-1), \sigma(k)] \text{ if } 0 < k < r, \\
    I^\sigma_r = (\sigma(r-1),\infty).
\end{gather*}
For $k \leq r$ and a function $f: [\omega]^r \to \omega$, we say that $f$ is \emph{$k$-trapped} if $f(\sigma) \in I^\sigma_k$ for all $\sigma \in [\omega]^r$; $f$ is \emph{trapped} if it is $k$-trapped for some $k$; and $f$ is \emph{properly trapped} if it is $k$-trapped for some $k < r$. $\FS^r$ can be restricted to a certain class of functions, so we may say \emph{$\FS^r$ for $k$-trapped functions}, etc.

\begin{lemma}\label{lem:FS.sca.trapped}
If $\FS^r$ for trapped functions has the (strong) cone avoidance property, then $\FS^r$ has the (strong) cone avoidance property.
\end{lemma}

\begin{proof}
We prove the lemma for the strong cone avoidance property. The proof for the cone avoidance property is similar and thus omitted.

Fix $A, B$ and $f: [\omega]^r \to \omega$ such that $A \not\leq_T B$. For each $\sigma \in [\omega]^r$, let
\begin{gather*}
    f_0(\sigma) = \min \{\sigma(0), f(\sigma)\}; \\
    f_k(\sigma) = \min\{\sigma(k), \max\{\sigma(k-1), f(\sigma)\}\} \text{ if } 0 < k < r; \\ f_r(\sigma) = \max \{\sigma(r-1) + 1, f(\sigma)\}.
\end{gather*}
By the assumption, we get $(H_k: k \leq r+1)$ such that
\begin{enumerate}
    \item $H_0 = \omega$ and $H_k \in [H_{k-1}]^\omega$ if $k > 0$;
    \item $A \not\leq_T B \oplus H_k$;
    \item if $k > 0$ then $H_k$ is free for $f_0, \ldots, f_{k-1}$.
\end{enumerate}

We claim that $H_{r+1}$ is free for $f$. Let $\sigma \in [H_{r+1}]^r$ be arbitrary. Then $f(\sigma) \in I^\sigma_k$ for some $k \leq r$ and thus $f(\sigma) = f_k(\sigma)$. As $H_{r+1}$ is free for $f_k$, $f(\sigma) \not\in H_{r+1} - \sigma$.
\end{proof}

So, it suffices to deal with $\FS^r$ for trapped functions. At first, it seems easier to deal with properly trapped functions. The properly trapped functions can be contained in compact $\Pi^0_1$ classes, and they look closer to finite colorings and thus may enjoy some properties derived from the strong cone avoidance property of the Achromatic Ramsey Theorem. In fact, compact $\Pi^0_1$ classes and the strong cone avoidance property of $\ART$ are key ingredients when we deal with the properly trapped functions later. However, it turns out that the $r$-trapped functions are the easiest to deal with.

\begin{lemma}\label{lem:FS.sca.r-trapped}
If $f: [\omega]^r \to \omega$ is $r$-trapped and $X$ is Martin-L\"{o}f random in $f$, then there exists an infinite $X$-computable $f$-free set.

Hence, $\FS^r$ for $r$-trapped functions has the strong cone avoidance property.
\end{lemma}

\begin{proof}
Fix $A, X$ and $f$ as in the assumption. We define a computable sequence of consecutive intervals as follows. Let $J_k =  [a_k, b_k] = [k,k]$ for $k < r$. Given $J_k = [a_k, b_k]$ defined and $k+1 \geq r$, let $a_{k+1} = b_k + 1$,
$$
    b_{k + 1} = \min \{b_k + 2^c: 2^c \geq 2^{k + 3} \binom{k+1}{r}\}
$$
and $J_{k+1} = [a_{k+1}, b_{k+1}]$. Let $c_k$ be such that $b_k - a_k = 2^{c_k} - 1$.

Let $T = \bigcup_{l < \omega} \prod_{k \leq l} J_k$. Then $T$ is a computably bounded computable subtree of $[\omega]^{< \omega}$. Moreover, $[T]$ can be computably mapped to $2^\omega$: the string $\sigma$ of length $r$ such that $\sigma(k) = k$ for all $k < r$ is mapped to the empty string; if $\sigma \in T$ of length $k \geq r$ is mapped to $\mu \in 2^{<\omega}$ and $x = a_k + i \leq b_k$, then $\sigma\<x\>$ is mapped to $\mu\nu$ where $\nu$ is the $i$-th element of $2^{c_k}$ under some computable enumeration of $2^{<\omega}$.

If $\sigma \in T \cap [\omega]^k$ is $f$-free, then
$$
  \{x \in J_{k}: \sigma\<x\> \text{ is \emph{not} free for } f\} \subseteq \{f(\rho): \rho \in [\sigma]^r\},
$$
as $f$ is $r$-trapped. So, for each $l$,
$$
    m_{T \cap [\omega]^{\leq l}} \{\sigma \in T \cap [\omega]^l: \sigma \text{ is free for } f\} > 2^{-1}.
$$
Let $S = \{\sigma \in T: \sigma \text{ is free for } f\}$. Under the above computable isomorphism between $[T]$ and $2^\omega$, $[S]$ is computably isomorphic to a $\Pi^f_1$ class of Cantor space of positive measure. By the relativization of a result of Ku\u{c}era (the corollary of Lemma 3 in \cite{Kucera:85}, see also \cite[Proposition 3.2.24]{Nies:2010.book}), $X$ computes an infinite path $Y \in [S]$ which is clearly free for $f$.

For the strong cone avoidance property, fix $A \not\leq_T B$. Then $A \not\leq_T B \oplus X$ almost everywhere in Cantor space. So we can pick $X$ and $Y$ such that $A \not\leq_T B \oplus X$, $X$ is Martin-L\"{o}f random in $f$, $Y$ is an infinite $f$-free set computable in $X$.
\end{proof}

Now, we can finish (F1).

\begin{corollary}\label{cor:FS.sca.1}
$\FS^1$ has the strong cone avoidance property.
\end{corollary}

\begin{proof}
By Lemmata \ref{lem:FS.sca.trapped} and \ref{lem:FS.sca.r-trapped}, we just need the strong cone avoidance property of $\FS^1$ for $0$-trapped functions, which follows easily from Theorems \ref{thm:DzJo} and \ref{thm:CGHJ}.
\end{proof}

Assume that $r > 1$ and $\FS^k$ for $k < r$ has the strong cone avoidance property. With these assumptions, we establish the cone avoidance property of $\FS^r$ and thus accomplish (F2).

\begin{lemma}\label{lem:FS.sca.ca}
$\FS^r$ has the cone avoidance property.
\end{lemma}

\begin{proof}
Let $X, Y$ and $g: [\omega]^r \to \omega$ be such that $X \not\leq_T Y \oplus g$.

For each $\sigma \in [\omega]^{r-1}$ and $x$, let $R_{\sigma,x} = \{y > \max \sigma: g(\sigma\<y\>) = x\}$. By the strong cone avoidance property of $\COH$, let $C \in [\omega]^\omega$ be such that $C$ is cohesive for $(R_{\sigma,x}: \sigma \in [\omega]^{r-1}, x < \omega)$ and $X \not\leq_T Y \oplus g \oplus C$. Thus, the following function is total:
$$
    \bar{g}(\sigma) =
    \left\{
      \begin{array}{ll}
        \lim_{y \in C} g(\sigma\<y\>), & \text{if } \lim_{y \in C} g(\sigma\<y\>) \text{ exists;} \\
        \max \sigma, & \text{otherwise.}
      \end{array}
    \right.
$$
By the induction hypothesis that $\FS^{r-1}$ has the strong cone avoidance property, let $D \in [C]^\omega$ be such that $X \not\leq_T Y \oplus g \oplus C \oplus D$ and $D$ is $\bar{g}$-free.

We define a desired $g$-free $H$ as a subset of $D$ by induction. Let $\xi_0 = \emptyset$. Suppose that $\xi_s \in [D]^{<\omega}$ is defined and free for $g$. By the cohesiveness of $C$, if $\sigma \in [\xi_s]^{r-1}$ and $y \in C$ is sufficiently large, then either $g(\sigma\<y\>) = \lim_{y \in C} g(\sigma\<y\>) = \bar{g}(\sigma)$, or $g(\sigma\<y\>) > \max \xi_s$. As $\xi_s$ is $\bar{g}$-free, if $\sigma \in [\xi_s]^{r-1}$ and $y \in C$ is sufficiently large, then either $g(\sigma\<y\>) = \bar{g}(\sigma) \not\in \xi_s - \sigma = \xi_s \<y\> - \sigma\<y\>$, or $g(\sigma\<y\>) > \max \xi_s$ and thus $g(\sigma\<y\>) \not\in \xi_s \<y\> - \sigma\<y\>$ as well. So the following number is defined:
$$
    x_s = \min \{y \in D: y > \max \xi_s \wedge \xi_s \<y\> \text{ is free for } g\}.
$$
Let $\xi_{s+1} = \xi_s \<x_s\>$. Finally, let $H = \bigcup_s \xi_s$. Then $H$ is $g$-free. Moreover, $H \leq_T g \oplus D$ and thus $X \not\leq_T Y \oplus g \oplus H$.
\end{proof}

Below, we work on (F3): to prove the strong cone avoidance property of $\FS^r$. By Lemmata \ref{lem:FS.sca.trapped} and \ref{lem:FS.sca.r-trapped}, it suffices to prove the following restriction of Theorem \ref{thm:FS.sca}.

\begin{lemma}\label{lem:FS.sca.k-trapped}
For $k < r$, $\FS^r$ for $k$-trapped functions has the strong cone avoidance property.
\end{lemma}

From now on, we fix $k < r$, $A \not\leq_T B$ and a $k$-trapped function $f: [\omega]^r \to \omega$. If $A \not\leq_T B \oplus X$, then $X$ is \emph{cone avoiding}; a Mathias condition $(\sigma,X)$ is \emph{cone avoiding} if $X$ is cone avoiding.

We prove Lemma \ref{lem:FS.sca.k-trapped} by constructing a cone avoiding infinite $f$-free set $G$. We build $G$ in two steps:
\begin{enumerate}
    \item We apply the strong cone avoidance property of the Achromatic Ramsey Theorem and of $\FS^q$ ($q < r$) to build a cone avoiding $E \in [\omega]^\omega$ such that
        \begin{enumerate}
            \item[(E)] for each $\rho \in [E]^{< \omega}$ with $k < |\rho| < r$, $f(\rho\tau) \not\in E - \rho$ for all $\tau \in [E]^{r - |\rho|}$ with $\min \tau$ sufficiently large.
        \end{enumerate}
    \item By a Seetapun-style Mathias forcing, we build a cone avoiding $f$-free $G \in [E]^\omega$. In this step, we need some measure theoretic argument, which could be taken as an application of a probabilistic method and is similar to that in Csima and Mileti \cite{Csima.Mileti:2009.rainbow}. The measure theoretic argument also needs the strong cone avoidance property of the Achromatic Ramsey Theorem.
\end{enumerate}

To facilitate the construction, for each $\sigma \in [\omega]^{<r}$, let $f_\sigma: [\omega]^{r - |\sigma|} \to \omega$ be such that $f_\sigma(\tau) = f(\sigma\tau)$. In particular, $f_\emptyset = f$. Moreover, fix $(d_n: n > 0)$ so that $\ART^n_{<\infty,d_n}$ has the strong cone avoidance property.

\subsection{The construction of $E$}

We build the desired $E$ from a cone avoiding $D$, which is sufficiently generic for Mathias forcing and has some nice properties.

\begin{lemma}
For each $\rho \in [\omega]^{<\omega}$ with $k < |\rho| < r$ and a cone avoiding $X \in [\omega]^\omega$, there exist $\theta \in [I^\rho_k]^{\leq d_{r-|\rho|}}$ and a cone avoiding $Y \in [X]^\omega$ such that $f_\rho([Y]^{r-|\rho|}) = \theta$.
\end{lemma}

\begin{proof}
As $f$ is $k$-trapped and $|\rho| > k$, $f_\rho$ is a finite coloring with range contained in $I^\rho_k$. So the lemma follows from the strong cone avoidance property of $\ART^{r - |\rho|}_{<\infty, d_{r - |\rho|}}$.
\end{proof}

By the above lemma and Lemma \ref{lem:M-forcing.ca}, we can build a descending sequence of cone avoiding Mathias conditions $((\sigma_n, X_n): n < \omega)$ and a sequence of finite sets $(\theta_\rho: k < |\rho| < r)$, which satisfy the following properties:
\begin{enumerate}
    \item for each $n$, $|\sigma_n| < |\sigma_{n+1}|$;
    \item for each $e$, there exists $n$ with $\Phi^B_e(Z) \neq A$ for all $Z \in (\sigma_n,X_n)$;
\end{enumerate}
and also
\begin{itemize}
    \item[(E')] if $k < |\rho| < r$, then $\theta_\rho \in [I^\rho_k]^{\leq d_{r-|\rho|}}$ and $f_\rho([X_n]^{r-|\rho|}) = \theta_\rho$ for some $n$.
\end{itemize}
Let $D = \bigcup_n \sigma_n$. Then $D$ is infinite and cone avoiding.

For each $l \in (k, r)$ and $i < d_{r-l}$, let $F_{l, i}: [\omega]^{l} \to \omega$ be such that
$$
  F_{l,i}(\rho) =
  \left\{
  \begin{array}{ll}
    \theta_\rho(i), & \text{if } i < |\theta_\rho|; \\
    0, & \text{otherwise}.
  \end{array}
  \right.
$$
By the induction hypothesis that $\FS^l$ for $l < r$ has the strong cone avoidance property, we can obtain a cone avoiding $E \in [D]^\omega$, which is $F_{l,i}$-free for all $l \in (k,r)$ and $i < d_{r-l}$.

\begin{lemma}
$E$ satisfies (E).
\end{lemma}

\begin{proof}
Fix an arbitrary $\rho \in [E]^{<\omega}$ with $l = |\rho| \in (k,r)$. As $E$ is $F_{l,i}$-free for all $i < d_{r-l}$, $\theta_\rho \cap (E - \rho) = \emptyset$. By (E'), there exists $b$ such that if $\tau \in [E \cap (b,\infty)]^{r-l}$ then $f_\rho(\tau) \in \theta_\rho$ and thus $f(\rho\tau) = f_\rho(\tau) \not\in E - \rho$. So $E$ satisfies (E).
\end{proof}

\subsection{The construction of $G$}

We build the desired $f$-free set $G$ as a subset of $E$, by Mathias forcing.

We work with a specific subset of Mathias conditions. A Mathias condition $(\sigma,X)$ is \emph{admissible} if
\begin{itemize}
    \item[(ad1)] $\sigma X \subseteq E$, $X$ is cone avoiding and
    \item[(ad2)] $\sigma\tau$ is $f$-free for all $\tau \in [X]^{r-k}$.
\end{itemize}
Since $E$ is cone avoiding and $\emptyset \tau = \tau$ is trivially $f$-free for all $\tau \in [E]^{r-k}$, $(\emptyset,E)$ is an admissible condition.

If $(\sigma,X)$ is admissible, then let $\mathcal{F}_{\sigma,X}$ be the set of all $k$-trapped $g: [\omega]^r \to \omega$ such that $\sigma\tau$ is $g$-free for all $\tau \in [X]^{r-k}$. By the definition of admissibility, $f \in \mathcal{F}_{\sigma,X}$. As each $k$-trapped $g$ satisfies $g(\rho) \leq \max \rho$ for all $\rho \in [\omega]^r$, $\mathcal{F}_{\sigma,X}$ can be identified with a $\Pi^X_1$ class in Cantor space.

By the lemma below, admissible conditions always capture some free sets.

\begin{lemma}\label{lem:FS.sca.M-condition}
If $(\sigma,X)$ is an admissible Mathias condition and $g \in \mathcal{F}_{\sigma,X}$, then there exists $Y \in [X]^\omega$ such that $\sigma Y$ is $g$-free. Moreover, if $g$ is cone avoiding then $Y$ can also be made cone avoiding.
\end{lemma}

\begin{proof}
For each $\rho \in [\sigma]^{< r}$, let $g_\rho$ be such that $g_\rho(\tau) = g(\rho\tau)$ for all $\tau \in [\omega]^{r - |\rho|}$ with $\min \tau > \max \rho$. By the Free Set Theorem, pick $Y \in [X]^\omega$ which is $g_\rho$-free for all $\rho \in [\sigma]^{< r}$; if $g$ is cone avoiding then $Y$ can also be made cone avoiding, by Lemma \ref{lem:FS.sca.ca}.

To show that $\sigma Y$ is $g$-free, fix an arbitrary $\xi \in [\sigma Y]^r$. Let $\rho = \xi \cap \sigma$ and $\tau = \xi \cap Y$.

We claim that $g(\xi) \not\in \sigma - \rho$. If $|\rho| < k$, then $k > 0$ and $g(\xi) \geq \xi(k-1) > \max \sigma$ as $g$ is $k$-trapped. Suppose that $|\rho| \geq k$. Then $\tau$ is contained in some $\tau' \in [X]^{r-k}$ and $\xi = \rho\tau \subseteq \sigma\tau'$. As $\sigma\tau'$ is $g$-free, $g(\xi) \not\in \sigma\tau' - \xi \supseteq \sigma - \rho$.

On the other hand, $g(\xi) \not\in Y - \tau$ as $Y$ is $g_\rho$-free and $g(\xi) = g_\rho(\tau)$.

So, $\sigma Y$ is free for $g$.
\end{proof}

As usual, we need to find a descending sequence of admissible Mathias conditions extending $(\emptyset, E)$. To extend an admissible condition to another, we have to satisfy the computability condition (ad1) and the combinatorial condition (ad2). We have seen how to satisfy such a computability condition in \S \ref{s:WRT.sca} and shall use a similar strategy later. To satisfy (ad2), we use a probabilistic method. For this probabilistic method, we introduce a class of trees on which the sequences suitable for (ad2) form a set of sufficiently large measure.

We fix $d = d_{r-k}$ and $c$ such that $2^{c-1} > d$. A finite tree $T \subset [\omega]^{<\omega}$ is \emph{fast growing of order $n$}, if for each $\tau \in T - [T]$,
$$
    |\{x: \tau\<x\> \in T\}| \geq 2^{|\tau| + c + 2} \binom{n+|\tau|}{k}.
$$
If $(\sigma,X)$ is admissible and $g \in \mathcal{F}_{\sigma,X}$, then let $\mathcal{T}(\sigma,X,g)$ be the set of all finite trees $T \subset [X]^{<\omega}$ such that $T$ is fast growing of order $|\sigma|$ and $\sigma\tau$ is $g$-free for each $\tau \in T$.

According to the following two lemmata, there is a sufficiently high probability of finding finite sequences on a fast growing tree to extend the finite part of an admissible condition.

\begin{lemma}\label{lem:FS.sca.fgt.cnt}
Suppose that $(\sigma,X)$ is admissible and $T \in \mathcal{T}(\sigma,X,f)$. Then there exists $b$ such that if $\tau \in T$, $\xi \in [E \cap (b,\infty)]^{r-k}$ and $\sigma\tau\xi$ is $f$-free then
$$
    |\{\tau\<x\> \in T: \sigma\tau\<x\>\xi \text{ is \emph{not} free for } f\}| \leq \binom{|\sigma\tau|}{k}.
$$
\end{lemma}

\begin{proof}
Let $a < \omega$ be a strict upper bound on all numbers occurring in $\sigma$ and $T$. By (E), pick $b > a$ such that for all $\rho \subseteq a$ and $\upsilon \subset E \cap (b,\infty)$, if $k < |\rho| < r$ and $|\rho\upsilon| = r$ then $f(\rho\upsilon) \not\in E - \rho$.

Fix $\tau \in T$ and $\xi \in [E \cap (b,\infty)]^{r-k}$ such that $\sigma\tau\xi$ is $f$-free.

\begin{claim}
If $\tau\<x\> \in T$ and $\zeta = \rho\<x\>\upsilon \in [\sigma\tau\<x\>\xi]^r$, then $f(\zeta) \not\in \sigma\tau\<x\>\xi - \zeta$.
\end{claim}

\begin{proof}
If $\upsilon = \emptyset$, then $f(\zeta) = f(\rho\<x\>) \leq \zeta(k) \leq x$ and $f(\rho\<x\>) \not\in \sigma\tau\<x\> - \rho\<x\>$, as $f$ is $k$-trapped, $k < r$ and $\sigma\tau\<x\>$ is $f$-free. So, $f(\zeta) = f(\rho\<x\>) \not\in \sigma\tau\<x\>\xi - \zeta$.

If $\upsilon \neq \emptyset$ and $\upsilon \neq \xi$, then $|\upsilon| < r - k$ and $|\rho\<x\>| > k$. As $\upsilon \subset E \cap (b,\infty)$, $f(\zeta) = f(\rho\<x\>\upsilon) \not\in E - \rho\<x\>$. As $f$ is $k$-trapped, $f(\zeta) \leq \zeta(k) \leq x$. Thus $f(\zeta) \not\in E - \zeta \supset \sigma\tau\<x\>\xi - \zeta$.

If $\upsilon = \xi$ then $|\rho\<x\>| = k$. As $f$ is $k$-trapped, $f(\zeta) \geq \zeta(k-1) = x$. So, $f(\zeta) \not\in \sigma\tau \supseteq \sigma\tau\<x\>\xi - \zeta$.
\end{proof}

By the above claim, if $\tau\<x\> \in T$ and $\sigma\tau\<x\>\xi$ is \emph{not} free for $f$, then $f(\zeta) \in \sigma\tau\<x\>\xi - \zeta$ for some $\zeta \in [\sigma\tau\xi]^r$. As $\sigma\tau\xi$ is $f$-free, $f(\zeta) \not\in \sigma\tau\xi - \zeta$. Thus $f(\zeta) = x$. As $f$ is $k$-trapped and $\max \sigma\tau < x < \min \xi$, $\zeta \cap \xi = \xi$. Hence,
$$
    \{x: \tau\<x\> \in T \wedge \sigma\tau\<x\>\xi \text{ is not free for } f\} \subseteq \{f(\rho\xi): \rho \in [\sigma\tau]^k\}.
$$
The lemma follows immediately.
\end{proof}

\begin{lemma}\label{lem:FS.sca.fgt}
Suppose that $(\sigma,X)$ is admissible and $T \in \mathcal{T}(\sigma,X,f)$. Then there exists $b$ so that if $\xi \in [E \cap (b,\infty)]^{r-k}$ and $\sigma\xi$ is $f$-free then
$$
    m_T \{\tau \in [T]: \sigma\tau\xi \text{ is \emph{not} free for } f\} \leq 2^{-c-1}.
$$
\end{lemma}

\begin{proof}
By the above lemma, for sufficiently large $b$, if $\xi \in [E \cap (b,\infty)]^{r-k}$, $\tau \in T - [T]$ and $\sigma\tau\xi$ is $f$-free, then
$$
    m_T \{\tau\<x\> \in T: \sigma\tau\<x\>\xi \text{ is \emph{not} free for } f\} \leq 2^{-|\tau| -c-2} m_T(\tau).
$$
The lemma follows immediately from the above inequality.
\end{proof}

Now, we can extend an admissible condition to force a cone avoiding requirement with the lemma below. The proof of the following lemma resembles the proof of Lemma \ref{lem:WRT.sca.G.ext}, in that we replace complicated questions concerning the specific $f$ by questions concerning all functions in $\mathcal{F}_{\sigma,X}$ and then exploit the cone avoidance property of $\WKL$.

\begin{lemma}\label{lem:FS.sca.Seetapun}
For each $e$ and each admissible $(\sigma,X)$, there exists an admissible $(\tau,Y) \leq_M (\sigma,X)$ such that $\Phi^B_e(Z) \neq A$ for all $Z \in (\tau,Y)$.
\end{lemma}

\begin{proof}
Let $\mathcal{U}$ be the set of $g \in \mathcal{F}_{\sigma,X}$ such that for every $T \in \mathcal{T}(\sigma,X,g)$,
$$
    m_T \{\upsilon \in [T]: \upsilon \text{ contains an $(e,B)$-splitting pair over } \sigma)\} < 2^{-1}.
$$
So, $\mathcal{U}$ is a $\Pi^{B \oplus X}_1$ subset of $\mathcal{F}_{\sigma,X}$.

\medskip

\emph{Case 1:} $\mathcal{U} = \emptyset$. In particular, $f \in \mathcal{F}_{\sigma,X} - \mathcal{U}$.

Fix $T \in \mathcal{T}(\sigma,X,f)$ such that
$$
    m_T \{\upsilon \in [T]: \upsilon \text{ contains an $(e,B)$-splitting pair over } \sigma)\} \geq 2^{-1}.
$$
Let $b$ be as in Lemma \ref{lem:FS.sca.fgt} for $(\sigma,X)$ and $T$, and let $Y_0 = X \cap (b,\infty)$.

For each $\xi \in [Y_0]^{r-k}$, let
$$
    h(\xi) = \{\upsilon \in [T]: \sigma\upsilon\xi \text{ is \emph{not} free for } f\}.
$$
So, $h$ is a finite coloring of $[Y_0]^{r-k}$. By Lemma \ref{lem:FS.sca.fgt}, each $h(\xi)$ is a subset of $[T]$ and $m_T h(\xi) \leq 2^{-c-1}$. By the strong cone avoidance property of the Achromatic Ramsey Theorem, there exist $\theta$ and $Y \in [Y_0]^\omega$ such that $|\theta| \leq d = d_{r-k}$, $Y$ is cone avoiding and $h(\xi) \in \theta$ for each $\xi \in [Y]^{r-k}$. Thus,
$$
    m_T \{\upsilon \in [T]: \exists S \in \theta (\upsilon \in S)\} \leq 2^{-c-1} d < 2^{-c-1} 2^{c-1} = 2^{-2}.
$$
Let $P = \{\upsilon \in [T]: \forall S \in \theta(\upsilon \not\in S)\}$. By the definition of $h$,
$$
    P = \{\upsilon \in [T]: \forall \xi \in [Y]^{r-k}(\sigma\upsilon\xi \text{ is free for } f)\};
$$
by the above inequality, $m_T P > 2^{-2} 3$.

So, we can pick $\upsilon \in [T]$ such that $\upsilon$ contains an $(e,B)$-splitting pair over $\sigma$ and $\sigma\upsilon\xi$ is $f$-free for all $\xi \in [Y]^{r-k}$. Fix $x$ and $\eta \subseteq \upsilon$ such that $\Phi^B_e(\eta; x) \downarrow \neq A(x)$. Let $\tau = \sigma\eta$. Then $(\tau,Y)$ is a desired admissible condition.

\medskip

\emph{Case 2:} $\mathcal{U} \neq \emptyset$.

By Theorem \ref{thm:JoSo} of Jockusch and Soare, pick a cone avoiding $g \in \mathcal{U}$. By Lemma \ref{lem:FS.sca.M-condition}, let $Y_0 \in [X]^\omega$ be such that $Y_0$ is cone avoiding and $\sigma Y_0$ is $g$-free. We define a $Y_0$-computable sequence of consecutive intervals by induction:
\begin{itemize}
    \item $J_0 = [a_0, b_0]$, where $a_0 = 0$ and $|Y_0 \cap J_0| =  2^{n_0}$ for some $2^{n_0} \geq 2^{c+2} \binom{|\sigma|}{k}$.
    \item if $J_l = [a_l, b_l]$ is defined, then $J_{l+1} = [a_{l+1}, b_{l+1}]$, where $a_{l+1} = b_l + 1$ and $|Y_0 \cap J_{l+1}| = 2^{n_{l+1}}$ for some $2^{n_{l+1}} \geq 2^{l+c+3} \binom{|\sigma|+l+1}{k}$.
\end{itemize}
For each $l$, let $T_l$ be the set of all $\upsilon \in [\omega]^{\leq l}$ such that $\upsilon(i) \in Y_0 \cap J_i$ for all $i < |\upsilon|$. Trivially, $T_l \in \mathcal{T}(\sigma,X,g)$. We can $Y_0$-computably map infinite paths of $\bigcup_l T_l$ to $2^\omega$: the empty string is mapped to the empty string; if $\sigma \in [T_l]$ is mapped to $\mu \in 2^{<\omega}$ and $x$ is the $i$-th element in $Y_0 \cap J_{l}$, then $\sigma\<x\>$ is mapped to $\mu \nu$ such that $\nu$ is the $i$-th element in $2^{n_{l}}$ (under some computable enumeration of $2^{<\omega}$).

As $g \in \mathcal{U}$, for each $l$,
$$
    m_{T_l} \{\upsilon \in [T_l]: \upsilon \text{ contains an $(e,B)$-splitting pair over } \sigma)\} < 2^{-1}.
$$
Let
$$
    T = \{\upsilon \in \bigcup_l T_l: \upsilon \text{ contains \emph{no} $(e,B)$-splitting pair over } \sigma)\}.
$$
Then under the above mapping, $T$ is $Y_0$-computably isomorphic to a $\Pi^{B \oplus Y_0}_1$ subset of Cantor space with positive measure. So, we can pick some $R$ such that $R$ is Martin-L\"{o}f random in $B \oplus Y_0$, $A \not\leq_T B \oplus Y_0 \oplus R$ and $Y_0 \oplus R$ computes some $Y \in [T]$. Then $Y$ is cone avoiding and contains no $(e,B)$-splitting pair over $\sigma$. It follows that $\Phi^B_e(Z) \neq A$ for all $Z \in (\sigma,Y)$. Moreover, as $(\sigma,X)$ satisfies (ad2) and $Y \subseteq X$, $(\sigma,Y)$ satisfies (ad2) as well and thus is admissible. Hence, $(\sigma,Y)$ is as desired.
\end{proof}

By an argument similar to Case 1 in the proof of Lemma \ref{lem:FS.sca.Seetapun}, we can extend the finite part of an admissible condition.

\begin{lemma}\label{lem:FS.sca.finite-ext}
Each admissible condition $(\sigma,X)$ admits an admissible extension $(\tau,Y)$ with $|\tau| > |\sigma|$.
\end{lemma}

\begin{proof}
Let $n = 2^{c} \binom{|\sigma|}{k}$ and $(x_i: i < n)$ be a strictly increasing sequence from $X$. Let $T$ be a finite tree, consisting exactly of $\emptyset$ and $\<x_i\>$ for $i < n$. Trivially, $T \in \mathcal{T}(\sigma,X,f)$. Let $b$ be as in Lemma \ref{lem:FS.sca.fgt.cnt} for $(\sigma,X)$ and $T$. For $\xi \in [X \cap (b,\infty)]^{r-k}$, let
$$
    h(\xi) = \{i: \sigma\<x_i\>\xi \text{ is \emph{not} free for } f\}.
$$
By Lemma \ref{lem:FS.sca.fgt.cnt}, $h(\xi)$ is a subset of $n$ with no more than $\binom{|\sigma|}{k}$  elements. By the strong cone avoidance property of $\ART$, fix $\theta$ and $Y \in [X \cap (b,\infty)]^\omega$ such that $|\theta| \leq d$, $Y$ is cone avoiding and $h([Y]^{r-k}) = \theta$. By the definition of $h$, $i \in S \in \theta$ if and only if $\sigma\<x_i\>\xi$ is not $f$-free for some $\xi \in [Y]^{r-k}$. So,
$$
    \{i < n: \forall S \in \theta(i \not\in S)\} = \{i < n: \forall \xi \in [Y]^{r-k}(\sigma\<x_i\>\xi \text{ is free for } f)\}.
$$
Let $N$ denote the set above. Then
$$
    |N| \geq n - |\theta| \binom{|\sigma|}{k} \geq (2^{c} - d) \binom{|\sigma|}{k} > 0.
$$
So, we can pick $i \in N$ and let $\tau = \sigma\<x_i\>$. Then $(\tau,Y)$ is as desired.
\end{proof}

With Lemmata \ref{lem:FS.sca.Seetapun} and \ref{lem:FS.sca.finite-ext}, we can get a descending sequence of admissible Mathias conditions $((\sigma_n, X_n): n < \omega)$ such that
\begin{enumerate}
  \item $(\sigma_0,X_0) = (\emptyset, E)$;
  \item $|\sigma_n| < |\sigma_{n+1}|$ for each $n$;
  \item for each $n$ and $Z \in (\sigma_{n+1},X_{n+1})$, $\Phi_n^B(Z) \neq A$.
\end{enumerate}
Let $G = \bigcup_n \sigma_n$. By admissibility, $G$ is $f$-free; by the above properties, $G$ is infinite and cone avoiding.

So, we have proved Lemma \ref{lem:FS.sca.k-trapped} and thus also Theorem \ref{thm:FS.sca}.

\section{Remarks and Questions}\label{s:Questions}

By \cite[Theorem 3.1]{Cholak.Jockusch.ea:2001.Ramsey}, there exists an $\omega$-model of $\RCA + \RT^2_2$ consisting of only $\Delta^0_3$ sets. On the other hand, Jockusch's bounds apply to most theorems in the four families. So, if $\Phi$ and $\Psi$ are theorems from the same family for exponents $2$ and $3$, respectively, then usually $\Phi \not\vdash \Psi$. Naturally, we expect this relation to generalize to larger exponents. In other words, we can ask whether any of the four families gives rise to a proper hierarchy of combinatorial principles below $\ACA$. Actually, this question has been asked in \cite{Cholak.Giusto.ea:2005.freeset, Csima.Mileti:2009.rainbow} for $\FS$, $\TS$ and $\RRT$, respectively. In \cite{Wang:2013.COH.RRT}, it is shown that $\RRT^3_2 \not\vdash \RRT^4_2$. Here we state the analogous question for $\ART$.

\begin{question}\label{q:hierarchy}
Fix $(d_k: 0 < k < \omega)$ as in \S \ref{s:WRT.sca}. Does $\ART^r_{c,d_r} \vdash \ART^{r+1}_{e,d_{r+1}}$ for any $r > 1$ and reasonable $c,e$?
\end{question}

A possible approach to answering the above questions would be to construct solutions which lie in some level of the $\low_n$ hierarchy, as the author did in \cite{Wang:2013.COH.RRT}. Recall that a set $X$ is \emph{$\low_n$} if $X^{(n)} \equiv_T \emptyset^{(n)}$; otherwise, $X$ is \emph{non-$\low_n$}. By \cite[Theorem 3.1]{Cholak.Jockusch.ea:2001.Ramsey}, $\RT^2_2$ admits non-$\low_2$-omitting, and so do all $\Phi$ in the four families for exponent $2$, as they are consequences of $\RT^2_2$; by \cite[Theorem 5.3]{Wang:2013.COH.RRT}, $\RRT^3_2$ admits non-$\low_3$-omitting. But in general, little is known.

\begin{question}
Given $\Phi$ which is a statement in the four families with exponent $r > 2$, does it admit non-$\low_r$-omitting?
\end{question}

For $\ART$, we can ask finer questions. Clearly, for fixed $r$ and $c$, if $d \leq e$ then $\ART^r_{c,d} \vdash \ART^r_{c,e}$. Dorais et al. \cite{Dorais.Dzhafarov.ea:2013} have established some implications between different instances of $\ART$, e.g., $\ART^{mn+1}_{k^n, k^n - 1} \vdash \ART^{m+1}_{k,k-1}$ (\cite[Proposition 5.3]{Dorais.Dzhafarov.ea:2013}), although they use a less artistic name for some instances of $\ART$ there. But in general, relations between distinct instances of $\ART$ are unkown.

\begin{question}
Compare distinct instances of $\ART$, e.g., $\ART^r_{c,d}$ and $\ART^r_{c,d+1}$.
\end{question}

Note that, if $c < \infty$ then $\ART^r_{c,d}$ is equivalent to $\ART^r_{d+1,d}$. However, the obvious proof for this equivalence cannot be generalized to yield $\ART^r_{d+1,d} \vdash \ART^r_{<\infty,d}$, even if $d$ is a standard positive integer.

\begin{question}
Compare $\ART^r_{d+1,d}$ and $\ART^r_{<\infty,d}$.
\end{question}

Another kind of metamathematical questions deals with the relation between the four families. Recently, Xiaojun Kang \cite{Kang:2013} proved that $\RRT^2_2 \not\vdash \TS^2$ and thus $\RRT^2_2$ is strictly weaker than $\FS^2$. The general picture is yet to be discovered.

\begin{question}
Compare theorems between different families.
\end{question}

It may also be of interest to investigate properties of the sequence of integers $(d_k: 0 < k < \omega)$ in \S \ref{s:WRT.sca}. By the proof of Theorem \ref{thm:WRT.sca}, we can take $d_k = S_{k-1}$ for $k > 0$, where $(S_i: i < \omega)$ is the sequence of Schr\"{o}der numbers (see \cite{Weisstein:Schroder.number}):
$$
  S_0 = 1, \ S_n = S_{n-1} + \sum_{k < n} S_k S_{n-k-1}.
$$

\begin{question}
Are Schr\"{o}der numbers optimal bounds for the Achromatic Ramsey Theorem to have the strong cone avoidance property?
\end{question}

Dorais et al. \cite[Proposition 5.5]{Dorais.Dzhafarov.ea:2013} have shown that $\ART^{r+1}_{2^r, 2^r - 1} \vdash \ACA$. But the gap between $2^r - 1$ and $S_r$ is quite large, as $S_r > 2^{2r-2}$.

\bibliographystyle{plain}

\begin{thebibliography}{10}

\bibitem{Cholak.Giusto.ea:2005.freeset}
Peter~A. Cholak, Mariagnese Giusto, Jeffry~L. Hirst, and Carl~G. Jockusch, Jr.
\newblock Free sets and reverse mathematics.
\newblock In {\em Reverse mathematics 2001}, volume~21 of {\em Lecture Notes in
  Logic}, pages 104--119. Assoc. Symbol. Logic, La Jolla, CA, 2005.

\bibitem{Cholak.Jockusch.ea:2001.Ramsey}
Peter~A. Cholak, Carl~G. Jockusch, and Theodore~A. Slaman.
\newblock On the strength of {R}amsey's theorem for pairs.
\newblock {\em Journal of Symbolic Logic}, 66(1):1--55, 2001.

\bibitem{Chong.Slaman.ea:2013.SRT}
Chitat Chong, Theodore~A. Slaman, and Yue Yang.
\newblock The metamathematics of stable {R}amsey's theorem for pairs.
\newblock {\em Journal of the American Mathematical Society}, to appear.

\bibitem{Chubb.Hirst.ea:2009}
Jennifer Chubb, Jeffry~L. Hirst, and Timothy~H. McNicholl.
\newblock Reverse mathematics, computability, and partitions of trees.
\newblock {\em Journal of Symbolic Logic}, 74(1):201--215, 2009.

\bibitem{Csima.Mileti:2009.rainbow}
Barbara Csima and Joseph Mileti.
\newblock The strength of the rainbow {R}amsey theorem.
\newblock {\em Journal of Symbolic Logic}, 74(4):1310--1324, 2009.

\bibitem{Dorais.Dzhafarov.ea:2013}
Fran\c{c}ois~G. Dorais, Damir~D. Dzhafarov, Jeffry~L. Hirst, Joseph~R. Mileti,
  and Paul Shafer.
\newblock On uniform relationships between combinatorial problems.
\newblock {\em Transactions of the American Mathematical Society}, to appear.

\bibitem{Dzhafarov.Jockusch:2009}
Damir~D. Dzhafarov and Carl~G. Jockusch, Jr.
\newblock Ramsey's theorem and cone avoidance.
\newblock {\em Journal of Symbolic Logic}, 74(2):557--578, 2009.

\bibitem{ErdHos.Hajnal.ea:1965.partition}
P.~Erd{\H{o}}s, A.~Hajnal, and R.~Rado.
\newblock Partition relations for cardinal numbers.
\newblock {\em Acta Mathematica Academiae Scientiarum Hungaricae}, 16:93--196, 1965.

\bibitem{Friedman:BRT}
Harvey Friedman.
\newblock {\em {B}oolean {R}elation {T}heory}.
\newblock manuscript.

\bibitem{Hirschfeldt.Shore:2007}
Denis~R. Hirschfeldt and Richard~A. Shore.
\newblock Combinatorial principles weaker than {R}amsey's theorem for pairs.
\newblock {\em Journal of Symbolic Logic}, 72(1):171--206, 2007.

\bibitem{Jockusch:1972.Ramsey}
Carl~G. Jockusch, Jr.
\newblock Ramsey's theorem and recursion theory.
\newblock {\em Journal of Symbolic Logic}, 37(2):268--280, 1972.

\bibitem{Jockusch.Soare:1972.TAMS}
Carl~G. Jockusch, Jr. and Robert~I. Soare.
\newblock {$\Pi \sp{0}\sb{1}$} classes and degrees of theories.
\newblock {\em Transactions of the American Mathematical Society}, 173:33--56, 1972.

\bibitem{Kang:2013}
Xiaojun Kang.
\newblock Combinatorial principles between {$\operatorname{RRT}^2_2$} and
  {$\operatorname{RT}^2_2$}.
\newblock {\em Frontier of Mathematics in China}, to appear.

\bibitem{Kucera:85}
Anton{\'{\i}}n Ku{\v{c}}era.
\newblock Measure, {$\Pi\sp 0\sb 1$}-classes and complete extensions of {${\rm
  PA}$}.
\newblock In {\em Recursion theory week (Oberwolfach, 1984)}, volume 1141 of
  {\em Lecture Notes in Math.}, pages 245--259. Springer, Berlin, 1985.

\bibitem{Lerman:83}
M.~Lerman.
\newblock {\em Degrees of Unsolvability}.
\newblock Perspectives in Mathematical Logic. Springer--Verlag, Heidelberg,
  1983.
\newblock 307 pages.

\bibitem{Liu:2012}
Jiayi Liu.
\newblock {$\text{RT}^2_2$} does not imply {$\text{WKL}_0$}.
\newblock {\em Journal of Symbolic Logic}, 77(2):609--620, 2012.

\bibitem{Nies:2010.book}
Andre Nies.
\newblock {\em {C}omputability and {R}andomness}.
\newblock Oxford Logic Guide. Oxford Univ. Press, 2010.

\bibitem{Seetapun.Slaman:1995.Ramsey}
David Seetapun and Theodore~A. Slaman.
\newblock On the strength of {R}amsey's theorem.
\newblock {\em Notre Dame Journal of Formal Logic}, 36(4):570--582, 1995.
\newblock Special Issue: Models of arithmetic.

\bibitem{Simpson:1999.SOSOA}
Stephen~G. Simpson.
\newblock {\em Subsystems of {S}econd {O}rder {A}rithmetic}.
\newblock Perspectives in Mathematical Logic. Springer-Verlag, Berlin, 1999.

\bibitem{Wang:2013.COH.RRT}
Wei Wang.
\newblock Cohesive sets and rainbows.
\newblock {\em Annals of Pure and Applied Logic}, 165(2):389--408, 2014.

\bibitem{Wang:RRT}
Wei Wang.
\newblock {R}ainbow {R}amsey {T}heorem for triples is strictly weaker than the
  {A}rithmetic {C}omprehension {A}xiom.
\newblock {\em Journal of Symbolic Logic}, 78(3):824--836, 2013.

\bibitem{Weisstein:Schroder.number}
Eric~W. Weisstein.
\newblock Schr\"{o}der {N}umber.
\newblock From MathWorld -- A Wolfram Web Resource.
\newblock http://mathworld.wolfram.com/SchroederNumber.html.

\end{thebibliography}

\end{document}